\documentclass[12pt,a4paper]{article}
\usepackage{amsmath,amsfonts,amsthm,amssymb,latexsym}
\usepackage{xspace}
\usepackage{xcolor}
\usepackage{enumitem}
\usepackage{graphicx}
\date{\today}

\newcommand{\F}{{\mathbb F}}
\renewcommand{\SS}{\textsf{SS}\xspace}
\newcommand{\RSS}{\textsf{RSS}\xspace}

\newcommand{\Soc}{\mathrm{Soc}}

\newcommand{\Out}{\mathrm{Out}}
\newcommand{\Sym}{\mathrm{Sym}}

\newcommand{\PSL}{\mathrm{PSL}}
\newcommand{\SL}{\mathrm{SL}}
\newcommand{\PSU}{\mathrm{PSU}}
\newcommand{\PSp}{\mathrm{PSp}}
\newcommand{\POm}{\mathrm{P}\Omega}
\newcommand{\Magma}{{\textsc {Magma}}\xspace}

\newtheorem{theorem}{Theorem} 
\newtheorem{lemma}[theorem]{Lemma} 

\newtheorem{notation}[theorem]{Notation}

\title{Minimal sized generating sets of permutation groups}
\author{Derek F. Holt and Gareth Tracey}
\begin{document}
\maketitle
\begin{abstract}
We present a randomised variant of an algorithm of Lucchini and
Thakkar~\cite{LucTha} for finding a smallest sized generating set in a
finite group, which has polynomial time expected running time in finite
permutation groups.
\end{abstract}
\section{Introduction}\label{sec:intro}
Lucchini and Thakkar prove in \cite{LucTha} that a generating set of smallest
possible size in a finite group $G$ can be found in time polynomial in the
group order $|G|$ and, in the case when $G$ is given as a subgroup of the
symmetric group $\Sym(n)$, this can be done using $n^2\lambda(G)^{13/5}$
\emph{generating tests}, where $\lambda(G)$ is the maximum of the
orders of the non-abelian composition factors of $G$, with $\lambda(G)=1$ for
solvable $G$. In this article we prove the following result.

\begin{theorem}\label{thm:main}
There is a randomised variant of the algorithm described in \cite{LucTha} for
subgroups of $\Sym(n)$ with expected running time polynomial in $n$.
The expected number of generating tests is $O(n^2 \log n)$, and
$O(n^2/\sqrt{\log n})$ if $G$ is transitive.
\end{theorem}

We have implemented this variant in \Magma~\cite{Magma} as a function
{\sc SmallestGeneratingSet}, and it runs fast in practice.

A generating test consists of testing whether a given subset $S$ of $G$ (which
in this paper will always have size bounded by the permutation degree $n$),
generates the whole of $G$. The \emph{Schreier-Sims} algorithm, which is easily
seen to run in polynomial time, can be used for this purpose. There is an
extensive discussion and detailed analysis of the complexity of variants of
this fundamental algorithm in \cite[Chapters 4 and 8]{Seress} and also in
\cite[Section 4.4.3]{HEO}.

For the proof of the theoretical result on expected running time, we assume
that a deterministic version of the Schreier-Sims algorithm, which we denote
by \SS, is used. There is however a randomised Monte-Carlo version of the
Schreier-Sims algorithm described in \cite[page 64]{Seress} (and in
\cite[Section 4.4.5]{HEO}) that often runs
much faster in practice, but has a small probability of returning an incorrect
negative answer; that is, it falsely reports that $\langle S \rangle \ne G$.
We shall call this version \RSS. In some (but not all) of the situations
in which we apply generating tests in the algorithm presented in this
paper, we repeatedly choose random elements in a search for certain
generators of $G$ that we know exist, and in these situations we can safely use
\RSS, because an incorrect negative answer would merely prolong the search.
Experiments with implementations
indicate that the increased speed that we gain from using \RSS rather than \SS
more than compensates for the time wasted by the occasional wrong answer. 
In our descriptions below we shall indicate in brackets when we can use \RSS
in an implementation, and when we must use \SS.

\section{Definitions and preliminary results}\label{sec:defs}
We shall assume that the reader is familiar with the results and proofs in
\cite{LucTha} and with the proof of Proposition 6 of that paper in particular.

The notion of \emph{$G$-equivalent} chief factors of $G$ is described in
\cite[Section 4]{LucTha} and, for a non-abelian chief factor $A$, the total
number of chief factors of $G$ that are $G$-equivalent to $A$ is denoted by
$\delta_G(A)$. There is a normal subgroup $R_G(A)$ of $G$ such that the socle
of $G/R_G(A)$ is the direct product of these $\delta_G(A)$ chief factors.

For a finite group $X$, let $P(X)$ be the smallest degree of a faithful
permutation representation of $X$; that is, the smallest $n$ for which $G$
embeds in $\Sym(n)$.
It is proved in \cite{KovPra} that, if $X$ is a finite group and $X/N$ has no
abelian normal subgroups, then $P(X/N) \le P(X)$. In particular, we have
$P(G) \ge P(G/R_G(A))$ for any non-abelian chief factor $A$ of $G$.
In \cite[Theorem 3.1]{EP} it is shown that, for a direct product
$N = S_1 \times \cdots \times S_r$ of finite non-abelian simple groups, we have
$P(N) = \sum_{i=1}^r P(S_i)$. Since $P(G)\geq P(N)$ whenever $N$ is a normal subgroup of a group $G$, we have the following result.

\begin{lemma}\label{mindegfac}
If $N \cong S^b$ is a non-abelian chief factor of the group $G$ with $S$ simple,
and $\delta_G(N)$ is as defined above, then $P(G) \ge b \delta_G(N) P(S)$.
\end{lemma}

We will also need a bound on the number of non-abelian chief factors in a transitive permutation group.
\begin{lemma}\label{lem:transchief}
Let $G\le S_n$. Then $G$ has at most $n$ non-abelian chief factors, and at most $\log{n}$ such factors if $G$ is transitive.
\end{lemma}
\begin{proof}
If $G$ is intransitive, then there is a partition $n_1+n_2=n$ with $n_1,n_2>1$ and groups $G_i\le S_{n_i}$ such that $G_1$ is transitive and $G$ is a subdirect subgroup in $G_1\times G_2$. In particular, the number of non-abelian chief factors of $G$ is at most the sum of the numbers of non-abelian chief factors of $G_1$ and $G_2$. Thus, since $\log(n_i)\le n_i$, the result will follow from the transitive case and induction.  

So assume that $G$ is transitive. When $G$ is primitive, the claim is \cite[Theorem 2.10]
 {Pyber}. So assume that $G$ is imprimitive, and let $\Sigma$ be a block system with $|\Sigma|$ minimal, so $G^\Sigma$ is primitive. Let $\Delta\in\Sigma$, and write $|\Delta|=r$ and $|\Sigma|=s$, so that $n=rs$. Let $R:=\mathrm{Stab}_G(\Delta)^{\Delta}$, and let $K$ be the kernel of the action of $G$ on $\Sigma$. Then by \cite[Lemma 5.8]{Tracey18}, $G$ embeds as a subgroup of $R\wr G^{\Sigma}$ in such a way so that if $X/Y$ is a non-abelian chief factor of $R$, then $(G\cap X^s)/(G\cap Y^s)$ is either trivial or a non-abelian chief factor of $G$. On the other hand, since $K:=G\cap R^s$, any chief series $1=X_0<\hdots<X_m=:R$ for $R$ yields a normal series $1=X_0^s\cap K<\hdots<X_m^s\cap K=K$ for $K$. Thus, all non-abelian chief factors of $G$ contained in $K$ are of the form $(G\cap X^s)/(G\cap Y^s)$ for $X/Y$ a non-abelian chief factor of $R$.
 It follows that if $b(G)$ denotes the number of non-abelian chief factors of $G$, then $b(G)\le b(R)+b(G^{\Sigma})$. The result then follows by induction.   
\end{proof}

We close this section with the following result from \cite{LucTha}.
\begin{lemma}\label{lem:LTgen}
Let $N=S^b$ be a minimal normal subgroup of a
finite group $G$, with $S$ non-abelian simple. Suppose that $G = \langle g_1,\ldots, g_d\rangle N$, with $d\geq 2$. Set $f(N):=90b|\Out(S)|/53$ and $s:=\lceil 1+\log_{|N|}(f(N)\delta_G(N))\rceil$. If $d\geq s$, then there exist $n_1,\hdots,n_d\in N$ such that $G = \langle g_1n_1,\ldots, g_sn_s,g_{s+1},\hdots,g_d\rangle$.
\end{lemma}
\begin{proof}
This is proved in \cite[Lemma 12]{LucTha} with $s$ replaced by $\lceil 8/5+\log_{|N|}(\delta_G(N))\rceil$. The invariant $\lceil 8/5+\log_{|N|}(\delta_G(N))\rceil$ comes from \cite[Proposition 6]{LucTha}. We can now see immediately from the final two displayed inequalities in the proof of \cite[Proposition 6]{LucTha} that one can replace $\lceil 8/5+\log_{|N|}(\delta_G(N))\rceil$ with our value of $s$ above.  
(In the notation of \cite[Proposition 6]{LucTha}, we are using the
penultimate inequality to replace $\log_2 |A|$ by $n|\Out (S)|$
in the final inequality,
and assuming that $\frac{53 |A|^{t-1}}{90 n |\Out (S)|} \ge \delta$
rather than $|A|^{t-8/5} \ge \delta$.)
\end{proof}

\section{The algorithm}\label{sec:algo}
Our algorithm involves the availability of uniformly distributed random
elements of a group $G$. There is an extensive discussion of this topic in
\cite[Section 3.2]{HEO}. But in this paper we are working in the context of
permutation groups with a known base and strong generating set, where group
elements have a normal form involving transversals of a stabiliser chain, as
described in \cite[Section 4.4.1]{HEO}. In this situation we can easily choose
random elements of the group by choosing random elements of these transversals,
which involves choosing at most $n$ random integers in the range
$[1\cdot\cdot\,n]$.  As we shall see, the expected number of random group
elements required in the algorithm is bounded by a polynomial function of $n$.
(The largest number arises in the choice of $O(n^2 \log n)$ sequences of random
elements $n_1,\ldots,n_d$ in the case $N$ non-abelian with $t'>d$.)

The input to the algorithm is a permutation group $G \le S_n$ defined by
a generating set.  We start by computing a chief series
$1 = N_0 < N_1 < ... < N_u = G$  of $G$. This can be done in polynomial time:
in fact in \emph{nearly linear} time~\cite[Section 6.2.7]{Seress}.
We have $u \le n-1$ by \cite[Theorem 10.0.5]{Menezes}.

Now we compute a generating set of the top factor $G/N_{u-1}$, which is simple.
If it is abelian then it is cyclic and any element of $G \setminus N_{u-1}$
generates $G$ modulo $N_{u-1}$. Otherwise it can be generated by two elements,
and by \cite[Theorem 3.1.7]{Menezes} a random pair of elements of $G$ has a
probability of at least $53/90$ of generating $G$ modulo $N_{u-1}$.  So we find
such a pair by repeatedly choosing random pairs and doing generating tests,
for which we can use \RSS.

In the main ``lifting'' step of the algorithm, we have found elements
$g_1,\ldots,g_d$ that generate $G$ modulo $N_k$ for some $k$ with
$1 \le k \le u-1$, and we want to lift them to a generating set of
$G$ modulo $N_{k-1}$.
\begin{lemma}\label{lem:liftingcases}\cite[Proposition 9 and the proof of Corollary 13]{LucTha}
 One of the following holds.
 \begin{description}
     \item[Lifting Case 1] there exist $n_1, \ldots,n_d$ in $N_k$ such that
$g_1n_1, \ldots,g_dn_d$ generate $G$ modulo $N_{k-1}$; or \item[Lifting Case 2]
there exist $n_1, \ldots,n_{d+1}$ in $N_k$ such that
$g_1n_1, \ldots,g_dn_d, n_{d+1}$ generate $G$ modulo $N_{k-1}$.
 \end{description}
 Note in particular that $d \le u+1-k$.
\end{lemma}

\begin{notation}\label{not:1}
Fix $1\le k\le u-1$ and set $N:=N_k/N_{k-1}$. Set $\delta:=\delta_{G/N_{k-1}}(N)$. As is observed in \cite{LucTha} there is no known efficient method of
calculating $\delta$, so in our algorithm below we use the possibly larger
number $\delta':=\delta'_{G/N_{k-1}}(N)$ which, as in \cite{LucTha}, we define to be the number of chief factors  of
$G/N_{k-1}$ having order $|N|$. This could potentially increase the running
time, but will not affect the correctness of the algorithm.
Finally, set
$t := \lceil{8/5+\log_{|N|}(\delta)}\rceil$
and $t' := \lceil{8/5+\log_{|N|}(\delta')}\rceil$.
Assume that we have found elements
$g_1,\ldots,g_d$ that generate $G$ modulo $N_k$ for some $k$ with
$1 \le k \le u-1$, and we want to lift them to a generating set of
$G$ modulo $N_{k-1}$.
\end{notation}

\textbf{\underline{The lifting step when $N$ is abelian}}. If $N$ is abelian and hence isomorphic to $C_p^l$ for some prime $p$ and $l>0$,
then we proceed as in \cite[Remark 11]{LucTha}. That is, we find a generating
set $e_1,\ldots,e_l$ of $N_k$ modulo $N_{k-1}$ (we can choose these from
the generators of $N_k$) and, for each $i=1,\ldots,d$ and
$j = 1,\ldots,l$, we test whether $g_1,\ldots,g_{i-1},g_ie_j,g_{i+1},\ldots,g_d$
generate $G$ modulo $N_{k-1}$. If so then we are in Lifting Case 1 and we are
done. If not, then by \cite[Proposition 9]{LucTha} we are in Lifting Case 2,
and furthermore $G$ is generated modulo $N_{k-1}$ by $g_1,\ldots,g_d,x$ for
any $x \in N_k \setminus N_{k-1}$.

\textbf{\underline{The lifting step when $N$ is non-abelian}}. If $N$ is non-abelian then $N \cong S^b$ for some non-abelian simple
group $S$ and $b \ge 1$. 

\noindent{\textbf{The case $\mathbf{d = 1}$.}} In this case, since $N$ is not cyclic, we know that we must be in Lifting
Case 2, so we keep testing random pairs $n_1, n_2$ (for which we can use \RSS in
the implementation) until $g_1n_1$ and $n_2$ generate $G$ modulo $N_{k-1}$ and,
as we shall demonstrate in the next section, the expected number of pairs that
we need test is at most $18n/53$.

\noindent{\textbf{The case $\mathbf{d>1}$.}} Let $t$ and $t'$ be as in Notation \ref{not:1}. There are two subcases.

\noindent\textbf{In the  case where $\mathbf{t' \le d}$} (and hence also $t\le d$), we are in Lifting Case 1 by
\cite[Lemma 12]{LucTha}, and we repeatedly test sequences $g_1n_1,\ldots,g_dn_d$
for random $n_1,\dots,n_d \in N$ (using \RSS in the implementation) until we
find the required generators. As we shall show in the next section, the
expected number of sequences that we need to choose is at most $18n/53$.

\textbf{So assume that $\mathbf{t'>d}$.} Then we do not know whether we are in Lifting Case 1 or 2.
In that case, we choose (at most) $90 d \delta' \log |N|/53$ random sequences
$n_1,\ldots,n_d$, and test them all (using \RSS in the implementation).
If we find generators $g_1n_1, \ldots ,g_dn_d$ of $G$ modulo $N_{k-1}$,
then we are done.

If not, then we do an exhaustive test of all $|N|^d$ such sequences using \SS.
If we still fail to find one then we know that we are in Lifting Case 2,
so we keep testing random sequences  $n_1, \ldots ,n_{d+1}$ (using \RSS in the
implementation) until we find one such that
$g_1n_1, \ldots ,g_dn_d, n_{d+1}$ generate $G$ modulo $N_{k-1}$.
The expected number of trials is again at most $18n/53$.

\section{Proof of polynomial running time}\label{sec:proof}

As in Section \ref{sec:algo}, the input to the algorithm is a permutation group $G \le S_n$ defined by
a generating set. We shall prove that the expected number of generating tests applied in the
algorithm that we have described is $O(n^2\log{n})$, and $O(n^2/\sqrt{\log n})$
when $G$ is transitive, thereby proving Theorem~\ref{thm:main}.
Many of the processes involved are easily seen to be
polynomial-time, usually linear in $n$, and will not require comment.
When we refer to the \emph{cost} of part of the algorithm, we shall mean the
(expected) number of generating tests required. The \emph{Lifting Cases} are those listed in Lemma \ref{lem:liftingcases}.

\noindent{\textbf{Set-up.}} We may assume that a chief series
$1 = N_0 < N_1 < ... < N_u = G$  of $G$ has been computed. Fix $1\le k\le u-1$ and set $N:=N_k/N_{k-1}$. Also, let $\delta:=\delta_{G/N_{k-1}}(N)$,
$\delta':=\delta'_{G/N_{k-1}}(N)$, $t := \lceil{8/5+\log_{|N|}(\delta)}\rceil$
and $t' := \lceil{8/5+\log_{|N|}(\delta')}\rceil$, as in Section \ref{sec:algo}. Suppose that we have found elements
$g_1,\ldots,g_d$ that generate $G$ modulo $N_k$ for some $k$ with
$1 \le k \le u-1$, and we want to lift them to a generating set of
$G$ modulo $N_{k-1}$.

\noindent\textbf{\underline{Cost of the lifting step when $N$ is abelian}.} When $N$ is abelian then as in Section \ref{sec:algo}, we can choose the elements $e_1,\ldots,e_l$ from from the elements of the generating set of $N_k$ returned by the chief
series algorithm. The chief series has length $u\le n-1$, so finding
$e_1,\ldots,e_l$ requires at most $n-1$ generating tests.
As explained in Section \ref{sec:algo}, we then apply up to $dl$ generating tests (for which we must use \SS) to
see if we are in Lifting Case 1.
It is proved in \cite[Theorem 1.2]{GlasbyEtAl} that a composition series of
a subgroup of $\Sym(n)$ with $s$ orbits has length at most $4(n-s)/3$, so the
sum of the values of $l$ arising from all elementary abelian chief factors is
at most $4(n-1)/3$. Since $d \le n-1$, the total cost of the algorithm
for all elementary abelian composition factors is at most $4(n-1)^2/3$
generating tests.
Moreover, when $G$ is transitive we have $d = O(n/\sqrt{\log n})$
by~\cite{LucMenMor}, and so the number of generating tests required is
$O(n^2/\sqrt{\log n})$.

\noindent\textbf{\underline{Cost of the lifting step  when $N$ is non-abelian}.} So assume now that $N$ is non-abelian. We first justify one of our key claims from Section \ref{sec:algo}.

\begin{lemma}\label{lem:Exp}
If we are in Lifting Case 1 [resp. Lifting Case 2], then the expected number of choices of random sequences $n_1,\hdots,n_d$ [resp. $n_1,\hdots,n_d,n_{d+1}$] of elements of $N_k/N_{k-1}$ needed to find a generating set $g_1n_1,\hdots,g_dn_d$ [resp. $g_1n_1,\hdots,g_dn_d,n_{d+1}$] of $G/N_{k-1}$ is at most $18n/53$.    
\end{lemma}
\begin{proof}
The proofs in Lifting Cases 1 and 2 are essentially identical, so assume that
we are in Lifting Case 1.  Recall from Section \ref{sec:defs} that there is a
normal subgroup $R_{G/N_{k-1}}(N)$ of $G/N_{k-1}$ such that the socle
of $\overline{G}:=(G/N_{k-1})/R_{G/N_{k-1}}(N)$ is the direct product of the
$\delta_{G/N_{k-1}}(N)$ chief factors of $G$ which are $(G/N_{k-1})$-equivalent
to $N$.

Assume first that $R_{G/N_{k-1}}(N)=1$, so that $G/N_{k-1}=\overline{G}$.
Then, as explained in \cite[Section 4]{LucTha}, we have $G/N_{k-1} \cong
L_\delta$, where $L$ is a monolithic primitive group as defined in
\cite[Section 3]{LucTha} with $A=\Soc(L)\cong N= N_k/N_{k-1}$. So the hypotheses of
\cite[Proposition 6]{LucTha} are satisfied by $L$ and $G/N_{k-1}$ with
$A \cong N= N_k/N_{k-1}$.  Our argument here uses the proof of that
proposition rather than the proposition itself and, since we are Lifing Case 1,
we are assuming that the conclusion of the proposition holds with $t=d$.

To help the reader follow the argument, we shall temporarily use the
notation of the proof of \cite[Proposition 6]{LucTha}, and so we replace
$G/N_{k-1}$ by $G = L_\delta$, and we shall use $g_1,\ldots,g_d$ to denote
their images in (the new) $G$. Recall that $G$ is a subgroup of $L^\delta$,
and $\Soc(G) \cong N^\delta$ and, as in the proof of the proposition,
we identify the normal subgroup $N$ of $G$ with the last component of
$\Soc(G)$, and $G/\Soc(G)$ can be identified with $L/A$.

There is a group $\Gamma$ acting semiregularly on a set $\Omega$
that consists of the sequences $(\bar{l}_1,\ldots,\bar{l}_d) \in L^d$ such
that $\langle \bar{l}_1,\ldots,\bar{l}_d \rangle = L$, and the image of
$\bar{l}_i$ in $L/N$ is equal to the image of $g_i$ in $G/A$ for
$1 \le i \le d$. For $g \in G \le L^\delta$ and $1 \le i \le \delta$, let
$g^{(i)} \in L$ be the projection of $g$ onto its $i$-th component.
Then the fact that the images of $g_1,\ldots,g_d$ generate $G/N$ implies that,
for $1 \le i \le \delta-1$, the $d$-tuples $(g_1^{(i)},\ldots,g_d^{(i)})$
lie in $\Omega$, and they lie in distinct orbits under the action of $\Gamma$.
Furthermore, for $n_1,\ldots, n_d \in N$, the elements
$g_1n_1,\ldots,g_dn_d$ generate $G$ if and only
$((g_1n_1)^{(\delta)},\ldots,(g_dn_d)^{(\delta)})$ is in $\Omega$
(i.e. if $\langle (g_1n_1)^{(\delta)},\ldots,(g_dn_d)^{(\delta)} \rangle = L$)
and, for $1 \le i \le \delta-1$,  does not lie in the same $\Gamma$-orbit as
$(g_1^{(i)},\ldots,g_d^{(i)})$.

Now, by \cite[Lemma 5]{LucTha} applied with $t=d$, the proportion of
$d$-tuples of elements $n_1,\ldots,n_d \in N$ for which
$((g_1n_1)^{(\delta)},\ldots,(g_dn_d)^{(\delta)}) \in \Omega$ is at least
$53/90$.  Furthermore, since we are assuming that there exist
$n_1,\ldots,n_d \in N$ with $G = \langle g_1n_1,\ldots,g_dn_d \rangle$, there
must be at least $\delta$ orbits of $\Gamma$ on $\Omega$.
Hence, since all such orbits  have the same size $|\Gamma|$,
among those $n_1,\ldots,n_d \in N$ for which
$((g_1n_1)^{(\delta)},\ldots,(g_dn_d)^{(\delta)})$ is in $\Omega$,
the proportion for which $G = \langle g_1n_1,\ldots,g_dn_d \rangle$ is at
least $1/\delta$. We conclude that, for random choice of $n_1,\ldots,n_d \in N$,
the probability that
$G = \langle g_1n_1,\ldots,g_dn_d \rangle$ is at least $53/(90 \delta)$.
Since $\delta \le n/5$ by Lemma~\ref{mindegfac}, this probability is at least
$53/(18n)$, and so the expected number of choices required to find a generating
set of $G$ is at most $18n/53$ as claimed. This completes the proof in the
case when $R_{G/N_{k-1}}(N)=1$, and we revert now to our original notation. 

Finally, if $R_{G/N_{k-1}}(N)\neq 1$, then the above shows that the probability that a random sequence $n_1,n_2,\ldots,n_d$ of elements of $N$ has the property that $\overline{G}=\langle \overline{g_1}\overline{n_1},\hdots,\overline{g_d}\overline{n_d}\rangle$ 
is at least $53/(18n)$. However, by \cite[Lemma 7]{LucTha},  $g_1n_1,\hdots,g_dn_d$ generates $G/N_{k-1}$ if and only if $\overline{g_1}\overline{n_1},\hdots,\overline{g_d}\overline{n_d}$ generates $\overline{G}$. The result follows.   
\end{proof}

We now follow the distinction of cases in Section \ref{sec:algo}.

\noindent{\textbf{The cost if $\mathbf{d=1}$ or $\mathbf{t'\le d}$.}}
As explained in Section \ref{sec:algo}, and applying Lemma \ref{lem:Exp}, the expected cost in these cases is at most $18n/53$. 

\noindent{\textbf{The cost if $\mathbf{t'> d}$.}}
Assume now that $t' > d$. Then, as explained in Section \ref{sec:algo}, we do not know whether we are in Lifting Case 1 or 2, so we perform at most $90d\delta'\log |N|/53$ generating sets to try and find generators $g_1n_1,\hdots,g_dn_d$ of $G$ modulo $N_{k-1}$. Now, $d < 13/5 + \log_{|N|}(\delta')$, so
$$90d\delta'\log |N|/53 < 234 \delta'\log |N|/53 + 90\delta'\log{\delta'}/53.$$ 
The sum of $\log |N|$ over all chief factors of $G$ is at most
$\log (n!) = O(n \log(n))$ so, since $\delta'\le n$ by Lemma \ref{lem:transchief},
the sum of the first of the terms on the right hand side of this inequality
over all chief factors of $G$ is $O(n^2 \log n)$.
Since there are at most $n$ such chief factors and $\delta' \le n$,
the same applies to the sum of the second terms.
So the total expected cost of tests of this type over all chief factors is
$O(n^2 \log n)$.
If $G$ is transitive, then $G$ has at most $\log{n}$ non-abelian chief factors
by Lemma~\ref{lem:transchief}, so the total expected cost we get above is
$O(n(\log{n})^2)$.

Suppose now that we have failed to find generators $g_1n_1,\hdots,g_dn_d$ of $G$ modulo $N_{k-1}$ via the generating tests above. We then perform an exhaustive test over all sequences $n_1,\ldots,n_d \in N^d$. We bound the cost of these exhaustive tests in each of the two possible cases:
\begin{enumerate}
\item We are in Lifting Case 1, but we fail to find the required elements
$n_1,\ldots ,n_d$. Then the expected cost is at most the probability that each
test using random $n_1,\ldots ,n_d$ fails, times the cost $|N|^d$ of the
exhaustive tests described at the end of Section~\ref {sec:algo}.
Since we perform $90d\delta'\log|N|/53$ tests using random $n_1,\ldots ,n_d$,
each of which has a probability
of at least $53/(18n) \le 53/(90\delta)$ of succeeding, and
$|N|^d = e^{d\log|N|}$, the expected cost is
(using $(1-1/x)^x < e^{-x}$ for $x>1$) at most
$$(1-53/(90\delta))^{90d\delta'\log|N|/53} e^{d\log|N|} < 1.$$
By Lemma \ref{lem:transchief}, this gives a
total expected cost over all chief factors of at most $n$.

\item We are in Lifting Case 2.
Let $f(N):=90b|\Out(S)|/53$. Then by Lemma \ref{lem:LTgen}, we have $d<\lceil 1 + \log_{|N|}(f(N)\delta)\rceil$.
By Lemma~\ref{mindegfac}, $n \ge b\delta P(S)$. Moreover,
$|\Out(S)| \le 2P(S)/3$ by Lemma~\ref{outmindeg} below,
so $f(N)\delta \le 60n/53$.

Thus, since we are in Lifting Case 2 and $d$ is an integer,
we have
$$d <1 + \log_{|N|}(f(N)\delta) \le 1+\log_{|N|}(60n/53)$$
so $|N|^{d-1} < 60n/53$,
and the cost $|N|^d$ of the tests described in the final paragraph of
Section~\ref{sec:algo} is less than $(60n/53)^{d/(d-1)}$.
\begin{align}\label{lab:est}
\text{This is at most $(60/53)^2n^2$ when $d=2$; and $(60/53)^{3/2}n^{3/2}$ when $d\geq 3$.}    
\end{align}

Let $m$ be the number of non-abelian chief factors for which Lifting Case 2
occurs. Then on the final occurrence we must have $d \ge 2+m-1$, and so
$2+m-1 < 1+\log_{|N|}(60n/53)$ and it is easily checked that this
yields  $m+1 < \log n$ for $n \ge 4$.
Then the total cost over all such factors is at most
\begin{align}\label{lab:est2}
O(\sum_{d = 2}^{m+1} n^{d/(d-1)}) < O(\sum_{d = 2}^{\log n} n^{d/(d-1)}) <
O(n^2 + n^{3/2}\log n) = O(n^2).    
\end{align}

Suppose now that $G$ is transitive. From \eqref{lab:est} and \eqref{lab:est2}, we see that to
verify the claimed bound of $O(n^2/\sqrt{\log n})$ generating tests, it
suffices to do this for a non-abelian factor with $d=2$.
In that case we have $|N| = |N|^{d-1} \le 60n/53$.
From Lemma 2.3.8\,(b) of \cite{GlasbyEtAl} we have $|\Out(S)|\le \log_2 |S|$
for all non-abelian finite simple groups $S$, and so
\begin{align*}
f(N) =\ & 90b|\Out(S)|/53 \le (90/53)b\log_2 |S| = (90/53)\log_2 |N| \le\\
&(90/53)\log _2 (60n/53).
\end{align*}
But $2=d<1 + \log_{|N|}(f(N)\delta)$ implies $|N| < f(N)\delta$, and
$\delta \le \log n$ by Lemma~\ref{lem:transchief}, so we get
$|N| = O((\log n )^2)$ and cost is $O(|N|^2) = O((\log n )^4)$
generating tests.
\end{enumerate}

Summing the cost estimates from each of the cases \textbf{$\mathbf{d=1}$ or $\mathbf{t'\le d}$} and $\mathbf{t'>d}$ above, and adding in our cost estimate from the case \textbf{$\mathbf{N}$ abelian}, we see that the total expected number of generating tests applied in the
algorithm that we have described is $O(n^2\log{n})$, and $O(n^2/\sqrt{\log n})$
when $G$ is transitive. This completes the proof of Theorem~\ref{thm:main}.

\section{Implementation and performance}\label{sec:performance}
Our \Magma implementation works for finite groups of matrices
(type {\textsf GrpMat} and solvable groups defined by a power-conjugate
presentation (type {\sf GrpPC}) as well as permutation groups
(type {\sf GrpPerm}) but since this article relates to permutation groups,
we shall use these only in our sample run-times below. We note however that
for soluble examples like $3^k:2$, it appears to be a little faster to work
with a group of type {\sf GrpPC}.

In the implementation
we deal separately first with cyclic groups, then with $p$-groups, and then
nilpotent groups. In these cases we can easily calculate $d(G)$ and find
generators as inverse images of a smallest sized generating set of
the abelian group $G/\Phi(G)$.

Otherwise, putting $k := \min(2,d(G/[G,G]))$, which we can easily calculate,
we know that $d(G) \ge k$, and we start by choosing a small number of
random subsets of $G$ of size $k$, and test (using \RSS) whether any of them
generate $G$. This enables us to find a smallest generating set very quickly
in a large number of groups.

If this fails, then we apply the algorithm that we have described in
Section~\ref{sec:algo}.
For the exhaustive test in the case when the chief factor $N$ is non-abelian,
we provide the option of computing the conjugation action of $N$ on the
set of $|N|^d$ candidates for generating sets, and only carrying out
generating tests on orbits representatives of this action. This results in
significant improvements in run-times.

Immediately after processing each chief factor we check whether the current
generators generate the whole group, and we can stop early if they do. This
results in significant variation from run to run of the times for the
examples in the table below, depending on whether or not we are lucky,
and find generators early.

We found that, for groups with $d(G)>2$ with both abelian and non-abelian
chief factors, if there is more than one chief series, then a series in
which the abelian factors are nearer the top of the group is likely
to result in faster run-times. This is because processing the abelian
factors is generally faster and, if the current generating set becomes
larger early in the complete calculation, then the exhaustive tests with
the non-abelian factors are less likely to be necessary.

Unfortunately, the \Magma\ {\sf ChiefSeries} intrinsic does the opposite,
and constructs series with the chief factors in the soluble radical at the
bottom. It is difficult to write an alternative version using existing
\Magma functionality, but we did manage a compromise in which the factors
in the non-abelian socle of the group are at the bottom. This results
in much faster times in the final three examples in the table below, where
the two reported times are using the default {\sf ChiefSeries} function, and 
our modified version.

It is not easy to find examples that take more than a few seconds to run.
For example the transitive groups of degree 48 almost all take less than
$0.1$ seconds.

The examples in the table below have been chosen because they are more
difficult, and some of them involve applying the exhaustive test for
a non-abelian chief factor. The split extensions in the table are \emph{crowns}:
that is, the complement of the specified normal subgroup acts in
the same way on each of the direct factors of the normal subgroup.
The two irreducible $4$-dimensional $\F_2A_5$-modules are denoted by $2^4_a$
and $2^4_b$, where the former is from the ntural module for $\SL(2,4)$
(which is isomorphic to $A_5$) and the latter is the deleted permutation
module.

\smallskip
\begin{tabular}{|r|r|r|r|}
\hline
Group & Degree & $d(G)$ & Time\\
\hline
$A_5^{19}$ & 95 &2& 0.6  \\
$A_5^{20}$ & 100 &3& 3 \\
$A_5^{100}$ & 500 &3& 252  \\
$L_3(2)^{57}$ & 399 &2& 75 \\
$L_3(2)^{58}$ & 506 &3& 236 \\
$L_3(2)^{100}$ & 700 &3& 350 \\
$L_3(2)^{100}:2$ & 800 &3& 513 \\
$A_6^{53}$ & 318 &2& 68 \\
$A_6^{54}$ & 324 &3&  804\\
$A_6^{100}$ & 600 &3&  1295\\
$3^{30}:2$ & 90 & 31 & 0.51\\
$3^{50}:2$ & 150 & 51 & 8\\
$3^{100}:2$ & 300 & 101 & 297\\
$(2^4_a)^{20}:A_5$ & 320 & 12 & 2.8\\
$(2^4_a)^{50}:A_5$ & 800 & 27 & 172\\
$(2^4_b)^{50}:A_5$ & 500 & 14 & 48\\
$((2^4_b)^5:A_5) \times A_5^{50}$ & 300 & 3 & 27, 16\\
$((2^4_b)^5:A_5) \times (L_3(2)^{58})$ & 456 & 3 & 259, 77\\
$(3^{10}:2) \times (L_3(2)^{58})$ & 436 & 11 & 224, 34\\
\hline
\end{tabular}

We did not attempt to implement the deterministic algorithm of Lucchini and
Thakkar described in \cite{LucTha} but, assuming that we used similar tricks to
those described above to deal with specific special cases, we would not expect a
significant difference in performance between our probabilistic and their
deterministic methods. Our results could be regarded as proving that the
expected performance of their algorithm is significantly faster than their
worst case analysis indicates.

\section{Appendix}
Lemma 2.3.8\,(a) of \cite{GlasbyEtAl} asserts that $P(S) > 2|\Out(S)|$ for all
finite simple groups $S$ except for $S=A_6$, when $P(S)=6$ and $|\Out(S)|=4$.
For the proof, the authors cite \cite[Lemma 2.7]{AschGur} which claims that the
result is true ``by inspection''. But in fact $S = \PSL(3,4)$ with $P(S)=21$
and $|\Out(S)| = 12$ is another exception to the general result, so we decided
to re-prove it here. Since $P(S)$ and $|\Out(S)|$ are known for all such $S$,
the proof is indeed a straightforward verification but, in view of the
error, it seems preferable to write it down.

In fact the only place in the paper~\cite{GlasbyEtAl}
where Lemma 2.3.8\,(a) of~\cite{GlasbyEtAl} is used
is in Case 2 (simple diagonal) of their proof of Theorem 1.6 (page 568)
where what they actually need is $|\Out(S)| < |S|/16$ for all non-abelian
simple groups $S$. This follows from Lemma 2.3.8\,(b) of \cite{GlasbyEtAl},
which (correctly) states that $|\Out(S)| \le \log_2 |S|$, and
$\log_2 |S| < |S|/16$ for $|S|>108$. So to complete the proof of
\cite[Theorem 6.1]{GlasbyEtAl} we need only observe
that $|\Out(A_5)| = 2 < |A_5|/16 = 3.75$.

In our proof of the following, we will use the estimates for primes $p$ and positive integers $e$: $p^{2e}>3e$ and $p^{3e}>6e$ for all $p$ and $e$; and $p^{e}+1 > 6e$ for $p^e>16$.  

\begin{lemma}\label{outmindeg}
Let $S$ be a finite simple group. Then $P(S)>3|\Out(S)|$ except when
$(S,P(S),|\Out(S)|) = (A_6,6,4)$, $(\PSL(3,4),21,12)$, $(A_5,5,2)$, or
$(\PSL(2,8),9,3)$.
\end{lemma}
\begin{proof}
It is well-known that $P(A_n) = n$ and $|\Out(A_n)|=2$ when $n \ne 6$, and
$|\Out(A_6)|=4$. For $S$ sporadic, we have $|\Out(S)| \le 2$ and $P(S) \ge 11$.
So the result holds in these cases.

For $S$ of Lie type, we refer to \cite[Page xvi]{Atlas} for $|\Out(S)|$,
and \cite[Table 4]{GuestEtAl} (which is derived from earlier results of
Mazurov, Vasilyev and others) for $P(S)$.
Let $q=p^e$ with $p$ prime in the following list of cases.

\begin{enumerate}
\item $S=\PSL(2,q)$ with $q \ge 7$ and $q \ne 9$:
$|\Out(S)| = \gcd(2,q-1)e$ and $P(S) = q+1$ (or $q$ when $q=7$ or $11$).
We check the result directly for $q \le 16$ and note that
$P(S) = p^e+1 > 6e \ge 3|\Out(S)|$ for $q>16$. 

\item $S=\PSL(d,q)$ with $d \ge 3$ and $(d,q) \ne (4,2)$
($\PSL(4,2) \cong A_8$): $|\Out(S)| = 2e\gcd(d,q-1)$ and $P(S)=(q^d-1)/(q-1)$.
For $d \ge 4$ and $(d,q) \ne (4,2)$,
we have $P(S) > q^{d-1} > 6eq > 3|\Out(S)|$ and
the same holds for $d=3$ (i.e. $q>6e$) when $q\not\in\{2,3,4,5,8,9,16\}$.
For those cases we can check that $P(S)=q^2+q+1>3|\Out(S)| = 6e\gcd(d,q-1)$
except for $S=\PSL(3,4)$.

\item $S=\PSU(d,q)$ with $d \ge 3$  and $(d,q) \ne (3,2)$:
$|\Out(S)| = 2e\gcd(d,q+1)$. When $d=3$ (and $q>2$) we have $P(S) = q^3+1>18e
\ge 3|\Out(S)|$ when  $q \ne 5$, and $P(S) = 50 > 18 = 3|\Out(S)|$ when $q=5$.
There are various cases for $P(S)$ when $d \ge 4$, but it is easily checked
that $P(S) > q^4 > 6(q+1)e \ge 3|\Out(S)|$ when $q > 2$, and
$P(S) \ge 27 > 3|\Out(S)|$ when $q=2$.

\item $S = \PSp(d,q)$ with $d \ge 4$ and $(d,q) \ne (4,2)$:
$|\Out(S)| = \gcd(d,q-1)e$. In all cases we have $P(S) \ge q^3>3qe >3|\Out(S)|$.

\item $S = \POm^\epsilon(d,q)$ with $d \ge 7$ and $d(q-1)$ even
(note that $\POm^\circ(d,q) \cong \PSp(d-1,q)$ when $d$ is odd and $q$ is
even): in all cases $|\Out(S)| \le 24 e$ with $|\Out(S)| \le 6$ when $q=2$,
and $P(S) > q^5 > 3|\Out(S)|$.

\item $S$ is an exceptional group of Lie type.
If $G = {^2}B_2(q)$ with $p=2$ and $e \ge 3$ odd, then
$P(S) = q^2+1 > 3e = 3|\Out(S)|$.
If $G = {^2}G_2(q)$ with $p=3$ and $e \ge 3$ odd, then
$P(S) = q^3+1 > 3e = 3|\Out(S)|$.
If $G = G_2(4)$ then $P(S)=416 > 6 = 3|\Out(S)|$.
In all other cases we have $|\Out(S)| \le 6e$, and
$P(S) > q^5 > 18e \ge  3|\Out(S)|$.
\end{enumerate}
\vspace*{-32pt}
\end{proof}

\end{document}